\documentclass[10pt]{elsarticle}
\usepackage[cp1251]{inputenc}
\usepackage[english]{babel}
\usepackage{amsmath}
\usepackage{amssymb}
\usepackage{amsfonts}
\usepackage{mathtools}
\usepackage{dsfont}
\usepackage{rotating}
\usepackage{graphicx}
\usepackage{floatflt,epsfig}
\usepackage{lineno,hyperref}
\usepackage{enumerate}
\usepackage{colortbl}
\usepackage{array,tabularx,tabulary,booktabs}
\usepackage{longtable}
\usepackage{multirow}
\usepackage{wrapfig}
\usepackage{subcaption}
\usepackage{caption}
\usepackage[table]{xcolor}
\newcolumntype{^}{>{\currentrowstyle}}

\journal{arXiv}
\newtheorem{lemma}{Lemma}

\newtheorem{theorem}{Theorem}
\newtheorem{corollary}{Corollary}

\begin{document}
\renewcommand{\abstractname}{Abstract}
\renewcommand{\refname}{References}
\renewcommand{\tablename}{Figure.}
\renewcommand{\arraystretch}{0.9}
\thispagestyle{empty}
\sloppy

\begin{frontmatter}
\title{Spectrum of the Transposition graph}

\author[03,01,02]{Elena~V.~Konstantinova}
\ead{e\_konsta@math.nsc.ru}

\author[01,02]{Artem Kravchuk}
\ead{artemkravchuk13@gmail.com}

\address[03] {Three Gorges Mathematical Research Center, China Three Gorges University, 8 University Avenue, Yichang 443002, Hubei Province, China}

\address[01]{Sobolev Institute of Mathematics, Ak. Koptyug av. 4, Novosibirsk 630090, Russia}
\address[02]{Novosibirsk State University, Pirogova str. 2, Novosibirsk, 630090, Russia}

\begin{abstract}
Transposition graph $T_n$ is defined as a Cayley graph over the symmetric group generated by all transpositions. It is known that all eigenvalues of $T_n$ are integers. However, an explicit description of the spectrum is unknown. In this paper we prove that for any integer $k\geqslant 0$ there exists $n_0$ such that for any $n\geqslant n_0$ and any $m \in \{0, \dots, k\}$, $m$ is an eigenvalue of $T_n$. In particular, it is proved that zero is an eigenvalue of $T_n$ for any $n\neq2$, and one is an eigenvalue of $T_n$ 
for any odd $n\geqslant 7$ and for any even $n \geqslant 14$. We also present exact values of the third and the fourth largest eigenvalues of $T_n$ with their multiplicities. 
\end{abstract}

\begin{keyword}
Transposition graph; integral graph; spectrum; 
\vspace{\baselineskip}
\MSC[2010] 05C25\sep 05E10\sep 05E15
\end{keyword}
% 05C25 graphs and groups
% 05E10 graphs and matrices
% 05E15 combinatorial problems concerning the classical groups;
\end{frontmatter}
%%%%%%%%%%%%%%%%%%%%%%%%%%%%%%%%%%%%%%%%%%%%%%%%%%%%%%%%%%%%%%%%%%%%%%%%%%%%%%%%%%%%%%%%%%%%%%%%%%%%%%%%%%%%%%%%%%%%%%%%%%%%%%%%%%

\section{Introduction}\label{sec1}

The {\it Transposition graph} $T_n$ is defined as a Cayley graph over the symmetric group $\mathrm{Sym}_n$ generated by all transpositions. The graph $T_n, n\geqslant 2$, is a connected bipartite $\binom{n}{2}$-regular graph of order $n!$ and diameter $(n-1)$~\cite{K08}. It is an edge--transitive graph but not distance--regular, and hence not distance--transitive graph. It was shown in~\cite{KL20} that the Transposition graph is integral which means that all eigenvalues of its adjacency matrix are integers~\cite{HS74}. Since $T_n$ is bipartite then its spectrum $Spec(T_n)$ is symmetric with respect to zero, where the spectrum of a graph is defined as a multiset of distinct eigenvalues together with their multiplicities~\cite{BH12}. Independently, an integerness of $T_n$ was shown in~\cite{KY97} along with finding the bisection width of the Transposition network $T_n$. More precisely, the following theorem was proved.

\begin{theorem}\label{KY-08} {\rm \cite[Lemma~3]{KY97}} The Transposition graph $T_n, n\geqslant 2,$ is an integral graph such that its largest eigenvalue is $\frac{n(n-1)}{2}$ with multiplicity $1$; its second largest eigenvalue is $\frac{n(n-3)}{2}$ with multiplicity $(n-1)^2$;  and for any $k, 3\leqslant k \leqslant n$, the value $\frac{n(n-2k+1)}{2}$ is an eigenvalue of $T_n$ with multiplicity at least $\frac{n!}{n(n-k)!(k-i)!}$.
\end{theorem}

This theorem, among other things, gives an idea on how the spectrum of the Transposition graph looks like. However, an explicit description of the spectrum is unknown. The next theorem gives an arrangement of eigenvalues around zero in the spectrum of this graph. 

\begin{theorem} \label{10}
For any integer $k\geqslant 0$, there exists $n_0$ such that for any $n\geqslant n_0$ and any $m \in \{0, \dots, k\}$, $m \in Spec(T_n)$. \end{theorem}

Thus, for large enough $n$, this result shows an existence of all integers in  $Spec(T_n)$ up to some upper bound. Moreover, since $T_n$ is bipartite hence $-m \in Spec(T_n)$. To prove this theorem we use basic facts from the representation theory of the symmetric group. We also prove new results on a correspondence between eigenvalues of the graph $T_n$ and partitions of $n$. These technical results are presented in Section~\ref{Sec2} along with definitions and notation. In particular, it is proved that zero is an eigenvalue of $T_n$ for any $n\neq2$, and one is an eigenvalue of $T_n$ for any odd $n\geqslant 7$ and for any even $n \geqslant 14$. Then we prove Theorem~\ref{10} in Section~\ref{Sec3}. Finally, in Section~\ref{Sec4} we estimate exact values of the third and the fourth largest eigenvalues of the graph $T_n$ and present their multiplicities.

\section{Preliminaries}\label{Sec2}

\subsection{Basic facts}

Let $G$ be a finite group with an identity element $1_G$, and $S$ be its generating set. Then a Cayley graph $\Gamma=Cay(G,S)$ is called {\it normal} if its generating set $S$ is closed under conjugation, i.~e. $S$ is the union of conjugacy classes of $G$.  The following theorem enables to compute eigenvalues and their multiplicities for any normal Cayley graph $\Gamma$ in terms of complex character values of $G$.  

\begin{theorem}\label{Z-88} {\rm \cite[Theorem~1]{Z88}} Let $G$ be a finite group with $s$ conjugacy classes and let $\{\chi_1,\chi_2,\ldots \chi_s \}$ be the set of all irreducible complex characters of $G$. Then the eigenvalues $\lambda_i, \ i=1,2,\ldots,s$, of any normal Cayley graph $\Gamma=Cay(G,S)$ are given by the following expression:

\begin{equation}\label{eigen_expr}
\lambda_i=\sum_{g\in S} \frac{\chi_i(g)}{\chi_i(1_G)},  
\end{equation}
and the multiplicity $mul(\lambda_i)$ of $\lambda_i$ is given by the formula:

\begin{equation}\label{mul_expr}
mul(\lambda_i)=\sum_{j=1,\\ \newline \lambda_j=\lambda_i}^{s} \chi_j(1_G)^2,
\end{equation}
\end{theorem}

It was shown in~\cite{KY97} that the conditions of Theorem~\ref{Z-88} hold for the Transposition graph $T_n=Cay(\mathrm{Sym}_n,T)$, where $T$ is the set of all transpositions. Moreover, the following useful expressions were obtained.

It is well-known fact (see~\cite{Sa01}) that there is one-to-one correspondence between the irreducible complex characters $\chi_i(g)$ and $\chi_i(1_G)$, $i=1,2,\ldots,p(n)$, of the symmetric group $\mathrm{Sym}_n$ and the partitions of $n$, where $p(n)$ is the number of partitions of $n$. Let a nonincreasing sequence $(n_1,n_2,\ldots,n_k), \ k\geqslant 1$, where $\sum_{j=1}^k n_j=n$, be the partition ${\bf i}=(n_1,\dots,n_k)\vdash n$ of $n$ corresponding to an irreducible complex character $\chi_i$. Then the following expression holds:
\begin{equation}\label{Formula3-KY}
\frac{\chi_i(\tau)}{\chi_i(I_n)}=\sum_{j=1}^k \frac{n_j(n_j-2j+1)}{n(n-1)}, 
\end{equation}
where $\tau$ is a transposition and $I_n$ is the identity permutation. Since the generating set $T$ of $T_n$ consists of $\frac{n(n-1)}{2}$ transpositions, then equations~(\ref{eigen_expr}) and~(\ref{Formula3-KY}) give an expression for an eigenvalue $\lambda_{\bf i}$ corresponding to the partition ${\bf i}$:

\begin{equation}\label{transp_eigen}
    \lambda_{\bf i} = \sum_{j=1}^k \frac{n_j(n_j-2j+1)}{2}. 
\end{equation}

Moreover, the last expression is bounded as follows:

\begin{equation}\label{eigen_ineq}
\sum_{j=1}^k \frac{n_j(n_j-2j+1)}{2} \leqslant \frac{(n-n_k)(n-n_k+1)}{2}+\frac{n_k(n_k-2k+1)}{2}.
\end{equation}

By Theorem~\ref{Z-88}, to compute multiplicities of eigenvalues of the Transposition graph $T_n$ we have to be able to compute $\chi_i(I_n)$. It is known (see~\cite{Sa01} for more details), this can be determined using a standard Young tableau associated with the partition of $n$ and the Frame-Robinson-Thrall hook-length formula defined as follows:

\begin{equation}\label{hook_formula}
\chi_i(I_n)=\frac{n!}{\prod_{t=1}^{k}\prod_{j=1}^{n_t} h_{tj}}, 
\end{equation}
where $h_{tj}$ is the hook-length of a box $(t,j)$ in a Young diagram. 

\subsection{New technical results}

We start with showing that the eigenvalue zero is in the spectrum of the graph $T_n$ for any $n\neq 2$. In what follows below, we use notation $(n_1,\dots,n_k,1\times t)$ for a partition in which $1$ appears $t$ times, where $t \geqslant 0$.

\begin{lemma}\label{lemma1}
For any odd $n\geqslant 1$, the partition $\left(\frac{n+1}{2}, 1\times\frac{n-1}{2}\right)$ corresponds to the eigenvalue zero of the Transposition graph $T_n$. For any even $n\geqslant 4$, the partition $\left(\frac{n}{2}, 2, 1\times \frac{n-4}{2}\right)$ corresponds to the eigenvalue zero of $T_n$.
\end{lemma}

\begin{proof} We prove the lemma by a direct substitution of partitions into the expression~(\ref{transp_eigen}). Indeed, if $n$ is odd then we have:
$$\lambda_{\left(\frac{n+1}{2},1\times\frac{n-1}{2}\right)}=\frac{1}{2}\left(\frac{n+1}{2}\cdot \left(\frac{n+1}{2}-2+1\right) \right) + \frac{1}{2}\sum\limits_{j=2}^{\frac{n-1}{2}+1}(1-2j+1)=$$ $$\frac{n^2-1}{8}-\frac{n^2-1}{8}=0,$$
and if $n$ is even then we have:
$$\lambda_{\left(\frac{n}{2},2,1\times\frac{n-4}{2}\right)}=\frac{1}{2}\cdot\frac{n}{2}\cdot\left(\frac{n}{2}-2+1\right)+\frac{1}{2}\cdot2\cdot(2-4+1)+\frac{1}{2}\sum\limits_{j=3}^{\frac{n-4}{2}+2}(1-2j+1)=$$
$$=\frac{n\cdot(n-2)}{8}-1-\frac{(n+2)\cdot(n-4)}{8}=0.$$

Note that partition $\left(\frac{n+1}{2},1\times\frac{n-1}{2}\right)$ holds for any odd $n\geqslant 1$, and the partition
$(\frac{n}{2},2,1\times \frac{n-4}{2})$ holds for any even $n\geqslant 4$. \hfill $\square$ \end{proof}

\begin{corollary} \label{cor-1}
In the spectrum of the Transposition graph $T_n$ there is the eigenvalue zero for any $n\neq2$.
\end{corollary}

\begin{proof} For $n=2$ there are only two partitions: $(2)$ and $(1,1)$. Substituting these partitions into the expression~(\ref{transp_eigen}) we have:
$$\lambda_{(2)}=\frac{1}{2} \cdot 2\cdot(2-2+1)=1 \neq 0,$$
and
$$\lambda_{(1,1)}=\frac{1}{2}\cdot 1 \cdot (1-2+1)+\frac{1}{2} \cdot 1 \cdot (1-4+1)=-1 \neq 0.$$
Thus, zero is not the eigenvalue of $T_n$ when $n=2$. However, by Lemma~\ref{lemma1} for any $n\geqslant 3$, we have $0\in Spec(T_n)$.  \hfill $\square$
\end{proof}\\

Similar result is obtained for the eigenvalue one. 

%Namely, we show that the eigenvalue one is in the spectrum of %the graph $T_n$ for any odd $n\geqslant 7$ and for any even %$n\geqslant 14$. More precisely, we have the following %statement. 

\begin{lemma}
For any odd $n\geqslant 7$, the partition $\left(\frac{n-1}{2},3,1\times \frac{n-5}{2}\right)$ corresponds to the eigenvalue one of the Transposition graph $T_n$. For any even $n \geqslant 14$, the partition $\left(\frac{n-6}{2},4,4,2,1\times \frac{n-14}{2}\right)$ corresponds to the eigenvalue one of $T_n$.
\end{lemma}

\begin{proof} We prove the lemma by a direct substitution of partitions into the expression~(\ref{transp_eigen}) such that if $n$ is odd then we have:
$$\lambda_{\left(\frac{n-1}{2},3,1\times \frac{n-5}{2}\right)}=\frac{1}{2}\cdot\frac{n-1}{2}\cdot\left(\frac{n-1}{2}-2+1\right)+ \frac{1}{2}\cdot3\cdot(3-4+1)+$$
$$\frac{1}{2}\cdot\sum\limits_{j=3}^{\frac{n-5}{2}+2}(1-2j+1)=\frac{(n-1)\cdot(n-3)}{8}+0-\frac{(n+1)\cdot(n-5)}{8}=\frac{8}{8}=1,$$
and if $n$ is even then we have:
$$\lambda_{\left(\frac{n-6}{2},4,4,2,1\times \frac{n-14}{2}\right)}=\frac{1}{2}\cdot\frac{n-6}{2}\cdot(\frac{n-6}{2}-2+1)+\frac{1}{2}\cdot4\cdot(4-4+1)+\frac{1}{2}\cdot4\cdot(4-6+1)+$$
$$+\frac{1}{2}\cdot2\cdot(2-8+1)+\frac{1}{2}\cdot\sum\limits_{j=5}^{\frac{n-14}{2}+4}(1-2j+1)=\frac{(n-6)\cdot(n-8)}{8}-5-\frac{n\cdot(n-14)}{8}=\frac{8}{8}=1.$$ 

Note that the partition $\left(\frac{n-1}{2},3,1\times \frac{n-5}{2}\right)$ holds for any odd $n\geqslant 7$, and the partition  $\left(\frac{n-6}{2},4,4,2,1\times \frac{n-14}{2}\right)$ holds for any even $n\geqslant 14$. \hfill $\square$
\end{proof}\\

The following two technical lemmas are used in Section~\ref{Sec3} to prove Theorem~\ref{10}. 

\begin{lemma}\label{lemma3}
If $n \geqslant 7$ is odd, the partition $(\frac{n-2\lambda+1}{2}, \lambda+2, 2\times (\lambda-1), 1\times \frac{n-4\lambda-1}{2})$ corresponds to the eigenvalue $\lambda \in \mathbb{N}$, where $1\leqslant \lambda \leqslant \frac{n-3}{4}$.
\end{lemma}

\begin{proof} By a direct substitution of the partition into the expression~(\ref{transp_eigen}) we immediately have:
$$\lambda_{\left(\frac{n-2\lambda+1}{2}, \lambda+2, 2\times (\lambda-1), 1\times \frac{n-4\lambda-1}{2}\right)}=$$ 
$$=\underbrace{\frac{1}{2}\left(\frac{n-2\lambda+1}{2}\right)\left(\frac{n-2\lambda+1}{2}-2+1\right)}_{(1)}+\underbrace{\frac{1}{2}(\lambda+2)(\lambda+2-2\cdot2+1)}_{(2)}+$$
$$+\underbrace{\frac{1}{2}\sum\limits_{j=3}^{\lambda+1}2(2-2j+1)}_{(3)}+\underbrace{\frac{1}{2}\sum\limits_{j=\lambda+2}^{\lambda+2+\frac{n-4\lambda-1}{2}-1}1(1-2j+1)}_{(4)},$$
where after calculations we have:
\begin{enumerate}[(1)]
    \item=$\frac{1}{8}(n-2\lambda+1)(n-2\lambda-2);$
    \item=$\frac{1}{2}(\lambda+2)(\lambda-1);$
    \item=$\sum\limits_{j=3}^{\lambda+1}2-2j+1=\left(\frac{3-6+3-2(\lambda+1)}{2}\right)(\lambda-1)=-(\lambda^2-1);$
    \item=$\sum\limits_{j=\lambda+2}^{\lambda+\frac{n-4\lambda-1}{2}+1}(1-j)=\left(\frac{1-\lambda-2+1-(\lambda+1+\frac{n-4\lambda-1}{2})}{2}\right)\frac{n-4\lambda-1}{2}=-\frac{(n+1)(n-4\lambda-1)}{8}.$
\end{enumerate}

Finally, putting all the members of the expression together we have:

$$\frac{1}{8}(n-2\lambda+1)(n-2\lambda-2)+\frac{1}{2}(\lambda+2)(\lambda-1)-(\lambda^2-1)-\frac{(n+1)(n-4\lambda-1)}{8}=\lambda.$$

It is easy to see that $\left(\frac{n-2\lambda+1}{2},\lambda+2,2\times(\lambda-1),1\times\frac{n-4\lambda-1}{2}\right)$ is a partition if and only if $\frac{n-2\lambda+1}{2}\geqslant \lambda+2$ and $\lambda\geqslant 1$. Therefore, if $\lambda \leqslant \frac{n-3}{4}$ then $\lambda \in Spec(T_n)$. Since $\lambda \geqslant 1$, this implies $\frac{n-3}{4} \geqslant 1$. Thus, $n \geqslant 7$ which completes the proof. \hfill $\square$
\end{proof}\\

\begin{lemma}\label{lemma4}
If $n\geqslant 14$ is even, the partition $(\frac{n-6\lambda}{2},2\lambda+2,\lambda+3,3\times (\lambda-1),2\times\lambda,1\times \frac{n-10\lambda-4}{2})$ corresponds to the eigenvalue $\lambda\in \mathbb{N}$, where $1\leqslant \lambda \leqslant \frac{n-4}{10}$.
\end{lemma}

\begin{proof} Using the same arguments as in the proof of Lemma~\ref{lemma3}, we substitute the partition into the expression~(\ref{transp_eigen}) and have:

$$\lambda_{\left(\frac{n-6\lambda}{2},2\lambda+2,\lambda+3,3\times(\lambda-1),2\times\lambda,1\times\frac{n-10\lambda-4}{2}\right)}=$$
$$=\underbrace{\frac{1}{2}\left(\frac{n-6\lambda}{2}\right)\left(\frac{n-6\lambda}{2}-2+1\right)}_{(1)} + \underbrace{\frac{1}{2}(2\lambda+2)(2\lambda+2-4+1)}_{(2)}+$$
$$+\underbrace{\frac{1}{2}(\lambda+3)(\lambda+3-6+1)}_{(3)} +
 \underbrace{\frac{1}{2}\sum\limits_{j=4}^{4+\lambda-2}3(3-2j+1)}_{(4)}+$$ 
$$+\underbrace{\frac{1}{2}\sum\limits_{j=\lambda+3}^{2\lambda+3-1}2(2-2j+1)}_{(5)} + 
 \underbrace{\frac{1}{2}\sum\limits_{j=2\lambda+3}^{2\lambda+3+\frac{n-10\lambda-4}{2}-1}1(1-2j+ 1)}_{(6)},$$
and after calculations we obtain:
\begin{enumerate}[(1)]
    \item=$\frac{1}{2}(\frac{n-6\lambda}{2})(\frac{n-6\lambda}{2}-2+1)=\frac{1}{8}(n-6\lambda)(n-6\lambda-2);$
    \item=$\frac{1}{2}(2\lambda+2)(2\lambda+2-4+1)=(\lambda+1)(2\lambda-1);$
    \item=$\frac{1}{2}(\lambda+3)(\lambda+3-6+1)=\frac{1}{2}(\lambda+3)(\lambda-2);$
    \item=$\frac{1}{2}\sum\limits_{j=4}^{4+\lambda-2}3(3-2j+1)=-\frac{1}{2}(3\lambda+6)(\lambda-1);$
    \item=$\frac{1}{2}\sum\limits_{j=\lambda+3}^{2\lambda+3-1}2(2-2j+1)=-(3\lambda+2)\lambda;$
    \item=$\frac{1}{2}\sum\limits_{j=2\lambda+3}^{2\lambda+3+\frac{n-10\lambda-4}{2}-1}1\cdot(1-2j+1)=-\frac{(n+2-2\lambda)(n-10\lambda-4)}{8},$
 \end{enumerate}   
for which a summation gives $\lambda$. 
%$\frac{1}{8}(n-6\lambda)(n-6\lambda-2)+(\lambda+1)(2\lambda-1)-\frac{1}{2}(\lambda+3)(\lambda-2)- -\frac{1}{2}(3\lambda + 6)(\lambda - 1) -(3\lambda + 2)\lambda-\frac{(n + 2 - 2\lambda)(n - 10\lambda -4)}{8} = \lambda$

Note that the expression $\left(\frac{n-6\lambda}{2},2\lambda+2,\lambda+3,3\times(\lambda-1),2\times\lambda,1\times\frac{n-10\lambda-4}{2}\right)$ is a partition if and only if $\frac{n-6\lambda}{2}\geqslant 2\lambda+2$ and $\lambda\geqslant 1$. Therefore, if $\lambda \leqslant  \frac{n-4}{10}$ then $\lambda \in Spec(T_n)$. Since $\lambda \geqslant 1$, this implies $\frac{n-4}{10} \geqslant 1$ which gives $n \geqslant 14$ and completes the proof. \hfill $\square$
\end{proof}

\section{Proof of Theorem~\ref{10}}\label{Sec3}

Let us choose $n_0$ such that 
\begin{equation}\label{n0}
    {\mathrm{min}}\left(\frac{n_0-3}{4},\frac{n_0-4}{10}\right)=k.
\end{equation}
If $k\geqslant 0$ then $n_0-3\geqslant 0$ and $n_0-4\geqslant 0$, hence~(\ref{n0}) is equivalent to $\frac{n_0-4}{10}=k$. Therefore, \begin{equation}\label{n01}
     n_0 = 10k + 4.
 \end{equation}

Now we prove that for any $n\geqslant n_0$ and for any $m\in\{0,\dots,k\}$, $m \in Spec(T_n)$. 

Since $n \geqslant n_0 \geqslant 4$ then by Lemma~\ref{lemma1} the eigenvalue zero is in the spectrum. 

If $n$ is odd then by Lemma~\ref{lemma3}, for any $1\leqslant m \leqslant \frac{n-3}{4}$, we have $m \in Spec(T_n)$. It follows from~(\ref{n01}) that $\frac{n-3}{4} \geqslant \frac{n_0-3}{4} \geqslant \frac{10k+4-3}{4}>m$ for any $m\in\{1,\dots,k\}$. Therefore, $m\in Spec(T_n)$ for any $m\in \{0,\dots,k\}$. 

If $n$ is even then by Lemma~\ref{lemma4}, for any $1\leqslant m \leqslant \frac{n-4}{10}$, we have $m \in Spec(T_n)$. It follows from~(\ref{n01}) that $\frac{n-4}{10}\geqslant \frac{n_0-4}{10} \geqslant \frac{10k+4-4}{10} \geqslant m$ for any $m\in \{1,\dots,k\}$. Again, we have that $m \in Spec(T_n)$ for any $m \in \{0, \dots, k\}$.

Thus, for any $n\geqslant n_0=10k+4$ and for any $m\in\{0,\dots,k\}$, $m\in Spec(T_n)$ which completes the proof of Theorem~\ref{10} 
\hfill $\square$

\section{The third and the fourth largest eigenvalues}\label{Sec4}

In this section we present exact values of the third and the fourth largest eigenvalues of the Transposition graph $T_n, n\geqslant 4,$ and their multiplicities.  

\begin{theorem} \label{3} The third largest eigenvalue of the Transposition graph  $T_n, n\geqslant 4,$ is $\frac{(n-1)(n-4)}{2}$ with multiplicity $\left(\frac{n(n-3)}{2}\right)^2$.
\end{theorem}

\begin{proof} We say that a partition ${\bf i_1}=(n_1,n_2,\dots,n_k)\vdash n$ is greater than a partition ${\bf i_2}=(m_1,m_2,\dots,m_l) \vdash n$, and write ${\bf i_1}>{\bf i_2}$, if an eigenvalue $\lambda_{\bf i_1}$ corresponding to $\bf i_1$ is greater than an eigenvalue $\lambda_{\bf i_2}$ corresponding to $\bf i_2$.  

By Theorem~\ref{KY-08}, the first and the second largest eigenvalues are $\frac{n(n-1)}{2}$ and $\frac{n(n-3)}{2}$, respectively. Moreover, these eigenvalues are associated with partitions $(n)$ and $(n-1,1)$, correspondingly~\cite{KY97}. Obviously, that  $(n)>(n-1,1)>(n-2,2)$. 

Our main goal now is to show that the partition $(n-2,2)$ is greater than any other partitions excepting $(n)$ and $(n-1,1)$, and it is the only partition associated with the third largest eigenvalue of $T_n, n\geqslant 4$. 

To show this, it is sufficient to prove that the following two inequalities hold:

\begin{equation}\label{ineq_part_2}
    (n-2,2)>(n-k,k)
\end{equation}
 for any $k>2$ and $k\leqslant \frac{n}{2}$, and

\begin{equation}\label{ineq_part_k}
    (n-2,2)>(n_1,\dots,n_k) \vdash n
\end{equation}
for any $k \geqslant 3$.

To prove~(\ref{ineq_part_2}), let us consider partitions $(n-k,k) \vdash n$ and  $(n-k-1,k+1) \vdash n$. Then, the following inequality  
$$(n-k,k)>(n-k-1,k+1)$$
holds if $n>2k+1$. Indeed, by~(\ref{transp_eigen}) we have to consider the inequality $(n-k)(n-k-2+1)+k(k-4+1)>(n-k-1)(n-k-1-2+1)+(k+1)(k+1-4+1),$ which gives $n>2k+1$ after reductions. Moreover, the condition $(n-k-1,k+1) \vdash n$ implies that $n-k-1 \geqslant k+1$ which again gives us $n\geqslant 2k+2>2k+1$.

Now let us show that inequality~(\ref{ineq_part_k}) holds for any $k \geqslant 3$. By~(\ref{transp_eigen}), we have the following expression for the eigenvalue corresponding to the partition $(n-2,2)$:
$$\lambda_{(n-2,2)}=\frac{(n-2)(n-2-2+1)}{2}+\frac{2(2-2\cdot 2+1)}{2}=\frac{(n-1)(n-4)}{2}.$$

Since $(n_1,\dots,n_k)\vdash n$ and $k\geqslant 3$, then $n_k \leqslant \frac{n}{3}$. Therefore, using the inequality (\ref{eigen_ineq}) leads to the following expression:
$$\lambda_{(n_1,\dots, n_k)} \leqslant \frac{1}{2}\left(\left(n-\frac{n}{3}\right)\left(n-\frac{n}{3}+1\right)+\frac{n}{3}\left(\frac{n}{3}-2\cdot 3+1\right)\right),$$
and finally we have:
$$\frac{(n-1)(n-4)}{2} > \frac{1}{2}\left(\left(n-\frac{n}{3}\right)\left(n-\frac{n}{3}+1\right)+\frac{n}{3}\left(\frac{n}{3}-2\cdot 3+1\right)\right),$$
which after reductions gives $(n-3)^2> 0$ holding for any $n \geqslant 4$.

Hence, taking into account Theorem~\ref{KY-08} and the inequalities~(\ref{ineq_part_2}),~(\ref{ineq_part_k}), for any $n \geqslant 4$, we have:
$$\lambda_{(n)}>\lambda_{(n-1,1)}>\lambda_{(n-2,2)}>\lambda_{\bf i},$$
%$$\lambda_{(n)}=\frac{n(n-1)}{2}>\lambda_{(n-1,1)}=\frac{n(n-3)}{2}>\lambda_{(n-2,2)}=\frac{(n-1)(n-4)}{2}>\lambda_{\bf i}$$
where ${\bf i} \in \{{\bf i_j}=(n_1,\dots,n_k)\vdash n \ | \  {\bf i_j} \notin \{(n),(n-1,1),(n-2,2)\}\}$.

Thus, it is shown that $\frac{(n-1)(n-4)}{2}$ is the third largest eigenvalue of $T_n$ associated with the partition $(n-2,2)$. Now let us get its multiplicity. 

It is easy to see that the hook-lengths of the partition $(n-2,2)$ are given as $h_{21}=2$, $h_{12}=n-1$, and $h_{1j}=n-j-1$ for any $j\in\{3,\ldots,n-2\}$. 
%such that we have a sequence $2,n-1,n-2,n-4,n-5,n-6,\dots,2,1$ 
Hence, by equations~(\ref{mul_expr}) and~(\ref{hook_formula}) we immediately have:
$${\rm mul}\left(\frac{(n-1)(n-4)}{2}\right) = \left(\frac{n!}{2\cdot(n-1)\cdot(n-2)\cdot(n-4)!}\right)^2 = \left(\frac{n(n-3)}{2}\right)^2,$$
which gives the multiplicity of the third largest eigenvalue and complete the proof. \hfill $\square$
\end{proof}

%\begin{figure}
%    \centering
%    \includegraphics[width=12cm]{hooks.eps}
%    \caption{Hook lengths for the standard Young %tableau corresponding to the partition $(n-2,2)$}
%    \label{hooks_fig}
%\end{figure}

\begin{theorem}\label{4} The fourth largest eigenvalue of the Transposition graph  $T_n, n > 6,$ is $\frac{n(n-5)}{2}$ with multiplicity $\left(\frac{(n-1)(n-2)}{2}\right)^2$.
\end{theorem}

\begin{proof} Our main goal is to show that the partition $(n-2,1,1)$ is greater than any other partitions excepting $(n),(n-1,1)$ and $(n-2,2)$. Moreover, it is the only partition associated with the fourth largest eigenvalue of $T_n, \ n>6$. To show this, it is sufficient to prove that the following three inequalities hold:
\begin{equation}\label{t4_1ineq}
    (n-2,1, 1)>(n_1, \dots, n_k),  
\end{equation}
for any $k\geqslant 4$ and $(n_1,\dots,n_k)\vdash n$, where $n>6$;
\begin{equation} \label{t4_2ineq}
    (n-2,1,1)>(n_1,n_2,n_3), 
\end{equation}
for any $(n_1,n_2,n_3)\vdash n$ and $(n_1,n_2,n_3)\neq (n-2,1,1)$;
\begin{equation}\label{t4_3ineq}
    (n-2,1,1)>(n_1,n_2), 
\end{equation}
if $n_1 \leqslant n-3$  for any $(n_1,n_2)\vdash n$, where $n>6$.

First, let us show that inequality~(\ref{t4_1ineq}) holds for any $k\geqslant 4$. By~(\ref{transp_eigen}), we have the following expression for the eigenvalue corresponding to the partition $(n-2,1, 1)$:
$$\lambda_{(n-2,1,1)}=\frac{(n-2)(n-2-2+1)}{2}+\frac{1(1-2\cdot 2+1)}{2}+\frac{1(1-2\cdot 3+1)}{2}=\frac{n(n-5)}{2}.$$

Since $(n_1,\dots,n_k)\vdash n$ and $k\geqslant 4$, then $n_k \leqslant \frac{n}{4}$. Therefore, using the inequality (\ref{eigen_ineq}) leads to the following expression:
\begin{align*}
\lambda_{(n_1,\dots,n_k)} \leqslant \sum\limits_{j=1}^k \frac{n_j(n_j - 2j + 1)}{2} \leqslant  \frac{(n-\frac{n}{4})(n-\frac{n}{4}-1)}{2}+\frac{\frac{n}{4}(\frac{n}{4}-2\cdot 4 + 1)}{2}= \\ =\frac{n \cdot (5n-20)}{16}.
\end{align*}
Comparing $\frac{n(n-5)}{2}$ and $\frac{n \cdot (5n-20)}{16}$ gives $3n>20$ which holds for any integer $n>6$.

To prove inequality~(\ref{t4_2ineq}), let us write an expression for the eigenvalue corresponding to the partition $(n_1, n_2, n_3)$:
\begin{equation}\label{t4_2_p}
    \begin{split}
      \lambda_{(n_1,n_2,n_3)}=\frac{n_1\cdot(n_1-2+1)}{2}+\frac{n_2\cdot(n_2-4+1)}{2}+\frac{n_3\cdot(n_3-6+1)}{2} \\
     = \frac{n_1\cdot(n_1-1)}{2}+\frac{n_2\cdot(n_2-3)}{2}+\frac{n_3\cdot(n_3-5)}{2}. 
   \end{split}
\end{equation}
Thus,~(\ref{t4_2ineq}) is equivalent to the following inequality: 
$$n^2-5 n>n_1\cdot(n_1-1)+n_2\cdot(n_2-3)+n_3\cdot(n_3-5)$$ for any $(n_1,n_2,n_3) \vdash n$. Moreover, since $n_3=n-n_1-n_2$, then we have:
$$n^2-5n>n_1\cdot(n_1-1)+n_2\cdot(n_2-3)+(n-n_1-n_2)(n-n_1-n_2-5),$$
which after calculations can be written as follows: 
$$2n_1^2+2n_2^2+4n_1+2n_2-2nn_1-2nn_2+2n_1n_2< 0$$
or as follows: 
$$n_1^2+n_2^2+2n_1+n_2-(nn_1+nn_2-n_1n_2)< 0.$$
If we rewrite the last inequality in the form:
$$n\cdot(n_1+n_2)>n_1^2+n_2^2+2n_1+n_2+n_1n_2,$$
then since $n=n_1+n_2+n_3$ we immediately have the inequality:
\begin{equation}\label{t_4_in2_p}
    n_1n_2+n_3(n_1+n_2)>2n_1+n_2.
\end{equation}

Let us show that the last inequality is true. Indeed, if $n_3>1$ then we have $n_1n_2+n_3(n_1+n_2)>n_1n_2+n_1+n_2,$ and since $n_2>1$ in this case we immediately have  $n_1n_2+n_1+n_2>n_1+n_1+n_2=2n_1+n_2$, which means that~(\ref{t_4_in2_p}) holds. If $n_3=1$, then~(\ref{t_4_in2_p}) is written as $n_1n_2>n_1$, and since $n_1>0$ then $n_2>1$. Thus,~(\ref{t_4_in2_p}) holds for all partitions of the form $(n_1,n_2,1)\vdash n$, where $n_2\geqslant 2$. If $n_2=1$ we have the partition $(n-2,1,1)$, and this completes a verification of~(\ref{t_4_in2_p}).

Now we prove inequality~(\ref{t4_3ineq}). Let us consider the expression~(\ref{transp_eigen}) corresponding to the partition $(n_1,n_2,0)$. It is the same as consider the partition $(n_1,n_2)$. Thus, if we prove inequality~(\ref{t_4_in2_p}) for $(n_1,n_2,0)$, then we show that~(\ref{t4_3ineq}) is true. Indeed, if $n_3=0$ then~(\ref{t_4_in2_p}) is rewritten as $n_1n_2>2n_1+n_2$, or as $n_1(n_2-2)>n_2$. Since $n_1\geqslant n_2$, then it holds for $n_2>3$. If $n_2=3$ we get $(n-3)(3-2)>3$. Hence, it is true for any $n>6$, which means that~(\ref{t4_3ineq}) holds.

Therefore, it has shown that the fourth largest eigenvalue is $\frac{n(n-5)}{2}$ and it is obtained due to the only partition $(n-2,1,1)$ for any $n>6$. Since the hook-lengths of $(n-2,1,1)$ are presented as $2,n,n-3,n-4,\dots,1$, then by~(\ref{mul_expr}) and~(\ref{hook_formula}) we immediately obtain its multiplicity as follows:
$${\rm mul}\bigg(\frac{n(n-5)}{2}\bigg)=\bigg(\frac{n!}{2\cdot n\cdot (n-3)!}\bigg)^2=\bigg(\frac{(n-1)\cdot(n-2)}{2}\bigg)^2,$$
which completes the proof. \hfill $\square$
\end{proof}

\section{Discussions and further research}

Despite we know something about the eigenvalues of the Transposition graphs, not so much is known about their multiplicities. In particular, there are no explicit formulas for multiplicities of the eigenvalues zero and one. Computational results on their multiplicities are presented in Table~1 and Table~2. 

As one can see from the Table~1, a behavior of multiplicities of the eigenvalue zero is quite unpredictable. Say, for $n=9$ its multiplicity is less than for $n=8$. We know that for a given $n$ multiplicities of eigenvalues depend on the number of partitions of $n$, and with growing $n$ the number of the corresponding partitions should grow as well. To understand this growing function for any eigenvalue in the spectrum, or even to find an approach for getting explicit formulas of multiplicities of the eigenvalues from Theorem~\ref{10} is one of the challenging problems. 

\vspace{5mm}

\begin{table}[h!]
\centering
\captionsetup[table]{labelformat=empty}
\begin{tabular}{ |c|c|c|c|c|c|c|c|c|c|c|} 
\hline
$n$ & 1 & 3 & 4 & 5 & 6 & 7 & 8 & 9 & 10 & 11 \\
%& 12 \\ 
\hline
${\rm mul(0)}$ & 1 & 4 & 4 & 36 & 256 & 400 & 9864 & 6664 & 790528 & 1474848 \\
%& 70669600 \\ 
\hline
\end{tabular}
\captionof{table}{Table 1: Multiplicities of the eigenvalue zero for any $n\leqslant 11$}
\end{table}

%\vspace{2mm}

\begin{table}[h!]
\centering
\captionsetup[table]{labelformat=empty}
\begin{tabular}{ |c|c|c|c|c|c|c|} 
\hline
$n$ & 7 & 9 & 11 & 13 & 15 & 17 \\
%& 19\\ 
\hline
${\rm mul(1)}$ & 441 & 46656 & 3052225 & 87609600 & {\small 2701400625} & {\small 3928998225152} \\
%& 389283476426208\\ 
\hline
\end{tabular}
\vspace{3mm}

\begin{tabular}{ |c|c|c|c|c|} 
\hline
$n$ & 14 & 16 & 18 & 20\\ 
\hline
${\rm mul(1)}$ & {\small 566130565} & {\small 301532774400} & {\small 274422662958600} & {\small 86181028874240000}\\ 
\hline
\end{tabular}
\captionof{table}{Table 2: Multiplicities of the eigenvalue one for some odd $n\geqslant 7$ and some even $n\geqslant 14$}
\end{table}

\section*{Acknowledgements}
The work of Artem Kravchuk was supported by the Mathematical Center in Akademgorodok, under agreement No. 075-15-2019-1613 with the Ministry of Science and High Education of the Russian Federation.

\end{document}